\documentclass[a4paper,12pt,reqno]{amsart}
\usepackage{amsthm, mathrsfs,amssymb,amsmath,}
\usepackage{enumerate}
\usepackage{mathtools}
\usepackage{breqn}
\makeatletter
\@namedef{subjclassname@2010}{%
	\textup{2010} Mathematics Subject Classification}
\makeatother

\usepackage{hyperref}
\allowdisplaybreaks
\numberwithin{equation}{section}

\usepackage{graphicx}
\usepackage{enumitem}

\newtheorem{theorem}{Theorem}[section]

\usepackage{xcolor}

\theoremstyle{definition}
\newtheorem{definition}[theorem]{Definition}
\newtheorem{example}[theorem]{Example}
\newtheorem{proposition}[theorem]{Proposition}

\newtheorem{note}[theorem]{Note}

\theoremstyle{remark}
\newtheorem{remark}[theorem]{Remark}

\numberwithin{equation}{section}

\begin{document}
	
	\title[Set-valued Riemann-Liouville integral and some regular selections]{Set-valued Riemann-Liouville integral and some regular selections}
	

		\author{Subhash Chandra$^{1,2}$}
	%   Address of record for the research reported here
	\address{$^1$Harish-Chandra Research Institute, A CI of Homi Bhabha National Institute, Prayagraj, India;
    \\ $^2$Faculty of Mathematics and Computer Science, Transilvania University of Braşov, Iuliu Maniu Street, nr. 50, 500091, Braşov, Romania}
	
	\email{sahusubhash77@gmail.com}
	\author{Syed Abbas$^{3,*}$}
\address{$^3$School of Mathematical and Statistical Sciences, Indian Institute of Technology Mandi\\ Kamand (H.P.)- 175005, India}
\email{sabbas.iitk@gmail.com, abbas@iitmandi.ac.in ( Email of corresponding author$^*$)}

	%\date{\today}
	
	
	\begin{abstract}
In this article, we introduce the notion of the Riemann-Liouville fractional integral of set-valued mappings via integrable selections. We establish fundamental properties of this set-valued fractional integral, including convexity, boundedness, and continuity with respect to the Hausdorff metric. The investigation of preservation of regularity under fractional integration with respect to the Hausdorff metric is given. We show that bounded variation and Lipschitz continuity of a set-valued mapping are inherited by its Riemann-Liouville fractional integral. One of the central aspects of set-valued analysis is the existence of regular selections. We establish such selections for the Riemann-Liouville fractional integral under corresponding regularity assumptions on the original mapping. In the scalar case, we further identify extremal selections given by the pointwise minimum and maximum of the fractional integral and show that they possess the same regularity properties. Finally, we show some interesting applications in differential inclusions.
%and directions for future research.
    
	\end{abstract}
	\maketitle

	\section{Introduction}
   The Riemann-Liouville integral is the first step towards defining fractional derivatives and hence fractional differential equations. This integral is a natural generalization of the regular n-fold integral of a suitable function. It is a fundamental operator in fractional calculus, which allows integration of non-integer values of the integration order $\alpha >0.$ This operator was developed through several contributions from mathematicians, most notably Liouville and Riemann, with recent formalizations \cite{pod}.  One can see that while classical calculus is limited to local, integer-order operations, the Riemann-Liouville integral provides non-local and memory-dependent behaviors. This enables us to model systems where past states influence the present, such as hereditary or anomalous processes. There are many theoretical and applied works in the field of fractional differential equations and other developments in fractional theory; for more details, we refer to 
   \cite{abbas2,mokh1,mokh2,jehad1,frac2,li2,VK,VK1,VK2,frac1} and references therein.

    The set-valued integrals and differential equations are very important in several areas of engineering and physical sciences. A class of set functional differential equations of neutral type is discussed in the paper \cite{lup1}, where authors establish existence and uniqueness results for the considered problem.  One natural question to ask is whether the Riemann-Liouville integral can be extended to a set-valued Riemann-Liouville integral. So, the set-valued Riemann-Liouville integrals play an important role in extending fractional calculus to tackle imprecisions, uncertainties, or multiple possibilities in mathematical modeling. In the set-valued context, it generalizes this to functions that map to sets (subsets of $  \mathbb{R}  $ or higher dimensions) rather than single points. Let $\mathcal{K}_c(\mathbb{R})$ be the family of all nonvoid compact convex subsets of $\mathbb{R}$. For a set-valued function $F: [a, b] \to \mathcal{K}_c(\mathbb{R})$, the Riemann-Liouville fractional integral of order $\alpha > 0$ is often defined selection-wise. It can also be defined via embeddings like the Aumann integral, but in interval-valued settings, it is computed component-wise. In the paper \cite{HBTS}, authors formally defined the interval-valued right-sided Riemann-Liouville fractional integral for an order \(\alpha > 0\). In the paper \cite{stef}, Stefanini and Bede introduc generalized Hukuhara differentiability, allowing the length of the interval solution to both grow and shrink over time. Shen in \cite{shen} discussed interval-valued fractional differential equations. 

In real-world systems, measurement mistakes, unpredictability, or inadequate data may make it impossible to determine parameters with precision. These are represented as sets via set-valued (including interval-valued) techniques. Set-valued analysis is a branch of mathematics that deals with functions mapping points to sets rather than single values. This setup is important in several areas such as optimization, control theory, differential inclusions, and economics. In recent times, the application in the area of reinforcement learning, where uncertainty or robustness leads to set-valued policies or dynamics has been explored for details, we refer to \cite{ergen}. Set-valued integrals, which apply classical integration to multifunctions, and selections, which select single-valued functions from a set-valued map while maintaining continuity or measurability properties, are key ideas. Such a setup enables fractional integrals to investigate dynamics in fields like rheology (material deformation), viscoelasticity (memory effects in materials), diffusion processes, control theory, and mathematical economics. When inputs are intervals rather than points, they allow for robust optimization. These integrals are key for solving Cauchy-type problems in fractional differential equations with set-valued right-hand sides, which is called differential inclusion. We provide the application of set-valued selection in differential inclusion. As it is more general and applicable, it comes with a cost in terms of additional restriction. The set-valued Riemann-Liouville integrals present several theoretical and practical difficulties due to operational complexities, existence-uniqueness issues, and integrability constraints. These challenges actually highlight the need for more research with relaxed assumptions and exploring broader set-valued extensions beyond intervals. Srivastava et.al. in \cite{sriv} discussed fractional-order inclusion with an interval-valued setting. Kara et.al. in \cite{kara} discussed inclusions for interval-valued functions. Huang et.al. in \cite{huang} proved the local and global existence of mild solutions for semi-linear fractional evolution inclusions governed by Riemann-Liouville derivatives in Banach spaces.
    
\par

Historically, set-valued integration emerged in the 1960s with foundational works by Aumann \cite{aumann1965integrals} and Debreu \cite{debreu}, building on earlier vector measure theory. This survey article gives a nice overview and important papers on this topic, focusing more on mathematical developments and especially connections to RL where relevant; for details, we refer to \cite{anca}. For set-valued selection, we refer to \cite{hess}. The foundational paper defining the Aumann integral via selections is given in \cite{aumann1965integrals}. The idea of  measures and integrals for multifunctions is discussed in \cite{art}. Readers interested more in set-valued integrals, selections, and viability theory may see \cite{aubin1}. Lupulescu in \cite{lup} introduced the first systematic framework for fractional calculus on interval-valued functions. Several important properties, such as monotonicity, linearity, and semigroup composition, have been established. Abbas et al. \cite{azam} discussed fractional differential inclusion with multipoint boundary conditions. Fixed-point techniques are utilized to study set-valued inclusions. Selection theory plays a fundamental role in the analysis of set-valued mappings, providing a bridge between multivalued models and single-valued representatives with desirable regularity properties. Given a set-valued mapping, the existence of measurable, continuous, or more regular selections allows one to reduce complex multivalued problems to classical functional frameworks while preserving essential qualitative features. Such selection results are indispensable in applications to differential inclusions, control theory, optimization, and viability theory, where solutions are often constructed through appropriately chosen selections. Classical measurable and continuous selection theorems, including those of Kuratowski–Ryll-Nardzewski and Michael, as well as their various refinements, have been extensively developed and applied; see, for example, \cite{aubin1,GM,VS}. Motivated by these developments, we introduce the theory of a set-valued Riemann-Liouville integral, and this is the new development in this direction. 

The paper is organized as follows. In Section 2, we recall the necessary preliminaries from set-valued analysis and introduce the left-sided set-valued
Riemann-Liouville fractional integral. In Section 3, we prove the main
properties of this operator, including convexity, nonemptiness, boundedness, compactness, continuity, preservation of bounded variation, and Lipschitz regularity. We also obtain regular selections and extremal selections for the fractional integral and discuss an application to fractional differential inclusions. In Section 4, we present some
concluding remarks and possible directions for future research.

\section{Preliminaries} \label{Prelim}

In this section, we lay out the background material related to this article.
\begin{definition}\cite{pod} 
Let $a,b\in\mathbb{R}$ with $a<b$, and let $f\in L^{1}([a,b])$. For $\alpha>0$, the left-sided Riemann-Liouville fractional integral
of order $\alpha$ of $f$ is defined, wherever the integral exists as a finite Lebesgue
integral, by
$${}_{a}J^{\alpha}f(u)
= \frac{1}{\Gamma(\alpha)}
\int_a^u (u-t)^{\alpha-1}f(t)\,dt,
\qquad u\in(a,b],$$
where $\Gamma$ denotes the Euler gamma function. 
\end{definition}
\begin{definition}\cite{aubin1}
Let $F:[a,b]\rightrightarrows \mathbb{R}$ be a set-valued mapping, then the set
$$Gr(F)=\left\{(u,w): u\in[a,b] \text{ and } w\in F(u)\right\},$$
is characterized as a graph of $F$. 
\end{definition}
\begin{definition}\cite{GB}
Let $\mathcal{K}(\mathbb{R})$ denote the family of all nonempty compact 
subsets of $\mathbb{R}$. For $A,B\in\mathcal{K}(\mathbb{R})$, the 
\emph{Hausdorff distance} between $A$ and $B$ is defined by
$$
H_d(A,B)
:=
\max\left\{
\sup_{a\in A}\inf_{b\in B} |a-b|,\;
\sup_{b\in B}\inf_{a\in A} |a-b|
\right\}.
$$
Since $A$ and $B$ are compact subsets of $\mathbb{R}$, $H_d$ is finite and 
defines a metric on $\mathcal{K}(\mathbb{R})$.
\end{definition}

\begin{definition}\cite{aubin1}
    Let $F: [a,b]\rightrightarrows\mathbb{R}$ be a set-valued mapping with nonempty
compact values. We say that $F$ is \emph{continuous} at $u_0\in[a,b]$ 
(with respect to the Hausdorff distance $H_d$) if 
$$
\lim_{u\to u_0} H_d\ \left(F(u),F(u_0)\right)=0.
$$
Equivalently, for every $\varepsilon>0$ there exists $\delta>0$ such that
$$
|u-u_0|<\delta
\quad \Longrightarrow \quad
H_d\!\left(F(u),F(u_0)\right)<\varepsilon .
$$
If this holds for every $u_0\in[a,b]$, then $F$ is said to be continuous on $[a,b]$ with respect to $H_d$.
\end{definition}
\begin{definition}\cite{aubin1}
Let $F:[a,b]\rightrightarrows\mathbb{R}$ be a set-valued mapping with 
nonempty compact values. We say that $F$ is \emph{Lipschitz on} $[a,b]$ 
(with respect to $H_d$) if there exists a constant $L\ge 0$ such that
$$
H_d\big(F(u),F(v)\big)
\le L\,|u-v|
\qquad\text{for all }u,v\in[a,b].
$$
Any such $L$ is called a \emph{Lipschitz constant} of $F$ (with respect 
to $H_d$).
\end{definition}
\begin{definition}\cite{SAVV}
Let $F:[a,b]\rightrightarrows\mathbb{R}$ be a set-valued mapping with 
nonempty compact values. For a partition 
$\mathcal{P}=\{a=u_0<u_1<\dots<u_n=b\}$ of $[a,b]$, define
$$
V(F,\mathcal{P})
:=
\sum_{i=1}^{n} H_d\big(F(u_i),F(u_{i-1})\big).
$$
The \emph{total variation} of $F$ on $[a,b]$ is
$$
V(F,[a,b])
:=
\sup_{\mathcal{P}} V(F,\mathcal{P}),
$$
where the supremum is taken over all partitions $\mathcal{P}$ of $[a,b]$.
We say that $F$ is of \emph{bounded variation} on $[a,b]$ 
(w.r.t.\ $H_d$) if $V(F,[a,b])<\infty$.
\end{definition}
\begin{definition}\cite{aumann1965integrals} A set-valued mapping $F:[a,b]\rightrightarrows \mathbb{R}$ is said to be non-negative if for every $u\in [a,b]$ and $w\in F(u)$, we have $w\geq 0.$
\end{definition}

\begin{definition}\cite{aubin1}\label{Bound}
A set-valued map $F:[a,b]\rightrightarrows \mathbb{R}$ with nonempty
values is said to be bounded, if it satisfies $$\sup_{u\in [a,b]}H_{d}(F(u),\{0\})<\infty.$$
\end{definition}
\begin{note}\label{nt2.8} \label{NOTE}
Since, 
\begin{align*}
   \sup_{u\in [a,b]}H_d(F(u),\{0\})=&\sup_{u\in [a,b]}\left\{ \max \left\{\sup_{x\in F(u)}\inf_{y\in \{0\}}\lvert x-y \rvert,~ \sup_{y\in \{0\}}\inf_{x\in F(u)}\lvert y-x \rvert\right\} \right\}\\
   =& \sup_{u\in [a,b]}\left\{ \max \left\{\sup_{x\in F(u)}\lvert x \rvert,~ \inf_{x\in F(u)}\lvert x \rvert\right\} \right\}\\
   =& \sup_{u\in [a,b]} \sup_{x\in F(u)} \lvert x \rvert.
\end{align*}
Therefore, in view of Definition \ref{Bound} a function $F$ is bounded if and only if,
\begin{equation}\label{Eq1}
  \sup_{u\in [a,b]}~\sup_{x\in F(u)} \lvert x \rvert < \infty.  
\end{equation}
\end{note}

\begin{remark}
In view of Note \ref{nt2.8}, one can observe that unlike the single-valued map, a constant set-valued map need-not be bounded. 

For example, consider $F:[0,1]\rightrightarrows \mathbb{R}$ is a set-valued constant map which is defined as $F(u)=[0,\infty)$ for each $u\in [0,1]$, then $F$ is not bounded as $\underset{u\in [0,1]}{\sup}~ \underset{x\in F(u)}{\sup}\lvert x \rvert=\infty.$ 
\end{remark}
\begin{definition}\cite{aubin1}
Let $F:[a,b]\rightrightarrows \mathbb{R}$ be a set-valued mapping with
nonempty values. A measurable function $f:[a,b]\to\mathbb{R}$ is called
a (measurable) selection of $F$ if
$$
f(t)\in F(t)\quad \text{for almost every } t\in[a,b].
$$
If in addition $f\in L^{1}([a,b])$, then $f$ is called an integrable selection of $F$.
\end{definition}
\begin{definition}\cite{aumann1965integrals}
A set-valued map, $F:[a,b]\rightrightarrows \mathbb{R}$ is said to be Borel-measurable if the graph, $Gr(F)$ is a Borel subset of $[a,b]\times \mathbb{R}$. Moreover, if there exists an integrable function, $h:[a,b]\rightarrow \mathbb{R}$ such that $\lvert x \rvert \leq h(u)$ for all $x\in F(u)$, then $F$ is said to be integrably bounded.
\end{definition}
\begin{example}
Let $F:[0,1]\rightrightarrows \mathbb{R}$ be a set-valued map defined as,
$F(u)=[-u,u]$ \\ $\text{ for } u\in [0,1].$
Then, the function, $h:[0,1]\rightarrow \mathbb{R}$ defined as $h(u)=u$  satisfies $\lvert x \rvert \leq h(u)$ for each $x\in F(u)$ and hence $F$ is an integrably bounded function.
\end{example}

\begin{remark}
Every bounded set-valued map is integrably bounded but the converse is not true.\\
Because, if $F:[a,b]\rightrightarrows \mathbb{R}$ is a bounded set-valued map, then by Note \ref{nt2.8} we get $\underset{u\in [a,b]}{\sup}\underset{x\in F(u)}{\sup}\lvert x \rvert = \mathcal{M}$ for some finite $\mathcal{M} \in \mathbb{R}$. Define $h:[a,b]\rightarrow \mathbb{R}$ such that 
$h(u)=\mathcal{M} \text{ for all } u\in [a,b].$
Then, $h$ being a constant function it is integrable and satisfies $\lvert x \rvert \leq h(u)$ for each $x\in F(u).$ Hence, $F$ is integrably bounded by $h$. Conversely, define $F:[0,1]\rightrightarrows\mathbb{R}$ by
$$
F(0)=\{0\},\qquad F(u)=[0,u^{-1/2}],\quad u\in(0,1].
$$
Then $F$ is integrably bounded by the function $h$ defined by
$h(0)=0$ and $h(u)=u^{-1/2}$ for $u\in(0,1]$, since
$$
|x|\le h(u)\quad\text{for all }x\in F(u),
\qquad
\int_0^1 h(u)\,du=2.
$$
However, $
\sup_{u\in[0,1]}\sup_{x\in F(u)}|x|=\infty,$
so $F$ is not bounded on $[0,1]$.
\end{remark}

\begin{definition}\cite{CK}
A subset $Y$ of a complete separable metric space $X$ is said to be analytic if $Y$ is a continuous image of a Borel subset of $X$.
A set-valued mapping $F:[a,b]\rightrightarrows \mathbb{R}$ is said to be analytic if its graph is an analytic subset of $[a,b]\times \mathbb{R}$.
\end{definition}

\begin{proposition}\cite{aumann1965integrals}\label{Prop1}
If $F:[a,b]\rightrightarrows \mathbb{R}$ is an analytic set-valued mapping with nonempty values, then there exists a Lebesgue measurable selection of $F$.
\end{proposition}
\begin{definition} \cite{lup}
Let $\mathbb{R}_{I}$ denote the family of all nonempty compact intervals of
$\mathbb{R}$. Let $G:[a,b]\to\mathbb{R}_{I}$ be an interval-valued function
such that $G(t)=[\underline{G}(t),\overline{G}(t)]$ and $G\in \mathrm{IR}([a,b])$, where $\mathrm{IR}([a,b])$ denotes the class
of interval Riemann integrable functions on $[a,b]$. For $\rho>0$, the
left-sided interval-valued Riemann-Liouville fractional integral of $G$ is
defined by
$$
J_{a+}^{\rho}G(u)
=
\frac{1}{\Gamma(\rho)}(\mathrm{IR})\int_a^u
(u-t)^{\rho-1}G(t)\,dt,~~ u>a, $$
where $\Gamma$ denotes the Euler gamma function.
\end{definition}

The interval-valued Riemann-Liouville fractional integral may be viewed as
a particular case of the set-valued framework considered below, since every
interval $G(t)=[\underline{G}(t),\overline{G}(t)]$ is a nonempty compact convex subset of $\mathbb{R}$.\\
More details concerning the inter-valued fractional calculus can be found in \cite{HBTS,Tuba}.\\\\
\textbf{Set-valued Riemann-Liouville fractional integral:} Here, we now introduce the Riemann–Liouville fractional integral of a set-valued mapping via integrable selections.
\begin{definition}
Let $F:[a,b]\rightrightarrows \mathbb{R}$ be a set-valued mapping and let $\mathfrak{F}$ denote the collection of all integrable selections of $F$. For $\rho>0,$ the set-valued left-sided Riemann-Liouville fractional integral of $F$ is defined by
$$_{a}\mathfrak{J}^{\rho}F(u)=\left\{\frac{1}{\Gamma(\rho)}\int_{a}^{u}(u-t)^{\rho-1}f(t)~\mathrm{d}t : f\in \mathfrak{F}\right\},~~u\in[a,b].$$ 
Here $\rho>0$ is a real number. For each $u\in[a,b]$, the set is taken over those
$f\in\mathfrak{F}$ for which the above integral exists as a finite Lebesgue integral.
\end{definition}
\begin{example}
Let $F:[0,1]\Rightarrow\mathbb{R}$ be defined by
$$F(t)=[0,t], \qquad t\in[0,1].$$
Then $F$ has nonempty compact convex values and is integrably bounded by
$h(t)=t$. For each $\lambda\in[0,1]$, the function
$f_{\lambda}(t)=\lambda t$
is an integrable selection of $F$. Hence, for $u\in[0,1]$,
$$\frac{1}{\Gamma(\rho)}\int_{0}^{u}(u-t)^{\rho-1}f_{\lambda}(t)\,dt
=
\lambda\,\frac{u^{\rho+1}}{\Gamma(\rho+2)}.$$
Since every integrable selection $f$ of $F$ satisfies $0\le f(t)\le t$, we get
$${}_{0}\mathfrak{J}^{\rho}F(u)
= \left[
0,\frac{u^{\rho+1}}{\Gamma(\rho+2)}
\right].$$
Hence, ${}_{0}\mathfrak{J}^{\rho}F(u)$ is nonempty for every
$u\in[0,1]$.
\end{example}
\begin{example}
Let $F:[0,1]\Rightarrow\mathbb{R}$ be defined by
$$ F(0)=\{0\}, \qquad F(t)=\left\{\frac{1}{t}\right\},\quad t\in(0,1]. $$
Then $F$ has nonempty compact values and is Borel measurable. However, every measurable selection $f$ of $F$ satisfies
$$
f(t)=\frac{1}{t}
\qquad\text{for almost every }t\in(0,1],
$$
and hence is not integrable on $[0,1]$.
 Therefore, $F$ has no integrable
selection, i.e. $\mathfrak{F}=\emptyset$. Consequently,
$$
{}_{0}\mathfrak{J}^{\rho}F(u)=\emptyset
\qquad\text{for every }u\in[0,1].
$$
\end{example}
\begin{remark}
The above examples show the importance of integrable selections in the
definition of the set-valued Riemann-Liouville fractional integral. In the
first example, integrable selections exist, and hence the fractional integral
is nonempty. In the second example, although each value of $F$ is nonempty
and compact, there is no integrable selection; therefore the corresponding
fractional integral is empty. This shows that nonempty compact values of $F$ alone do not
guarantee the existence of an integrable selection.
\end{remark}
\section{Main Results}

In this section, we establish some fundamental properties of the set-valued Riemann–Liouville fractional integral, which play a crucial role in the development of our main results. 
\begin{theorem} \label{convex}
Let $F:[a,b]\rightrightarrows\mathbb{R}$ be a set-valued mapping with the property that $F(t)$ is convex for each $t\in[a,b]$.
Then for every $\rho>0$ and every $u\in[a,b]$, the set ${}_{a}\mathfrak{J}^{\rho}F(u)$ is convex.
\end{theorem}
\begin{proof}
Fix $u\in[a,b]$. If ${}_{a}\mathfrak{J}^{\rho}F(u)$ is empty or a singleton, there is nothing to prove. Otherwise, take $y_{1},y_{2}\in {}_{a}\mathfrak{J}^{\rho}F(u)$. By the definition of ${}_{a}\mathfrak{J}^{\rho}F(u)$, there exist integrable selections $f$ and $g$ of $F$ such that
$$
y_{1}=\frac{1}{\Gamma(\rho)}\int_{a}^{u}(u-t)^{\rho-1}f(t)dt,
\qquad
y_{2}=\frac{1}{\Gamma(\rho)}\int_{a}^{u}(u-t)^{\rho-1}g(t)dt.
$$
Let $\lambda\in[0,1]$ and define $h=\lambda f+(1-\lambda)g$. Since $f(t),g(t)\in F(t)$ for a.e.~$t$ and each $F(t)$ is convex, we have $h(t)\in F(t)$ for a.e.~$t$. Moreover, $h$ is integrable because it is a linear combination of integrable functions. Hence $h$ is an integrable
selection of $F$, and therefore
$$
\frac{1}{\Gamma(\rho)}\int_{a}^{u}(u-t)^{\rho-1}h(t)dt
\in {}_{a}\mathfrak{J}^{\rho}F(u).
$$
By linearity of the Lebesgue integral,
$$
\frac{1}{\Gamma(\rho)}\int_{a}^{u}(u-t)^{\rho-1}h(t)dt
=
\lambda y_{1}+(1-\lambda)y_{2}.
$$
Thus $\lambda y_{1}+(1-\lambda)y_{2}\in {}_{a}\mathfrak{J}^{\rho}F(u)$ for all $\lambda\in[0,1]$, proving that ${}_{a}\mathfrak{J}^{\rho}F(u)$ is convex. This completes the proof.
\end{proof}

\begin{theorem}\label{THnonempty}
Let $F:[a,b]\rightrightarrows\mathbb{R}$ be a set-valued mapping with nonempty values. Assume that $F$ is Borel-measurable and
bounded on $[a,b]$. Then for every $\rho>0$ and every $u\in[a,b]$, the set ${}_{a}\mathfrak{J}^{\rho}F(u)$ is nonempty.
\end{theorem}
\begin{proof}
Since $F$ is Borel-measurable with nonempty values, its graph is an
analytic subset of $[a,b]\times\mathbb{R}$. By Proposition \ref{Prop1},
there exists a Lebesgue-measurable selection $f:[a,b]\to\mathbb{R}$
such that $f(t)\in F(t)$ for a.e. $t\in[a,b]$.

Since $F$ is bounded on $[a,b]$, there exists $M>0$ such that
$$
\sup_{x\in F(t)}|x|\le M
\quad\text{for all }t\in[a,b].
$$
Hence $|f(t)|\le M$ for a.e. $t\in[a,b]$, so $f$ is integrable
and therefore $f\in\mathfrak{F}$.

Fix $\rho>0$ and $u\in(a,b]$. Then
$$
\frac{1}{\Gamma(\rho)}
\int_a^u (u-t)^{\rho-1}|f(t)|dt
\le
\frac{M}{\Gamma(\rho)}
\int_a^u (u-t)^{\rho-1}dt
=
\frac{M(u-a)^\rho}{\Gamma(\rho+1)}
<\infty.
$$
Thus, the fractional integral of $f$ is well defined and finite.
By definition, it belongs to ${}_{a}\mathfrak{J}^{\rho}F(u)$,
so ${}_{a}\mathfrak{J}^{\rho}F(u)\neq\emptyset$.
At $u=a$, the integral equals zero and
${}_{a}\mathfrak{J}^{\rho}F(a)=\{0\}$.
This completes the proof.
\end{proof}
\begin{theorem}
Consider $F:[a,b]\rightrightarrows \mathbb{R}$ being a set-valued mapping with nonempty
compact values, Borel measurable and integrably bounded on $[a,b]$, and
let $\rho>0$. If $F$ is bounded on $[a,b]$, then the set-valued
Riemann-Liouville integral
$$
u \longmapsto {}_{a}\mathfrak{J}^{\rho}F(u)
$$
is also bounded on $[a,b]$, that is,
$$
\sup_{u\in[a,b]} H_d\!\left({}_{a}\mathfrak{J}^{\rho}F(u),\{0\}\right) < \infty.
$$
\end{theorem}
\begin{proof}
By using Theorem \ref{THnonempty}, $\mathfrak{F}$ is non-empty and since $F$ is bounded, therefore from Note \ref{NOTE}, we have 
$$\sup_{u\in [a,b]} \sup_{x\in F(u)} \lvert x \rvert < \infty.$$
Hence there exists a constant
$$\mathcal{M}=\sup_{t\in[a,b]}\sup_{x\in F(t)}|x|<\infty.$$

Let $f\in\mathcal F$ be an integrable selection of $F$. Then
$f(t)\in F(t)$ for almost every $t\in[a,b]$, and therefore
$|f(t)|\le \mathcal{M} \quad \text{for almost every } t\in[a,b].$ Since $|f(t)|\le \mathcal{M}$ for almost every $t\in[a,b]$, and the Riemann-Liouville integral is defined via the Lebesgue integral, we may use this bound inside the integral.
\begin{align*}
  H_d\left(~_{a}\mathfrak{J}^{\rho}F(u), \{0\} \right)=& \sup_{u\in [a,b]} \sup_{y\in~_{a}\mathfrak{J}^{\rho}F(u)} \lvert y \rvert\\
  =& \sup_{u\in [a,b]} \sup_{f\in \mathfrak{F}} \left\lvert \frac{1}{\Gamma(\rho)} \int_{a}^{u}(u-t)^{\rho-1}f(t)~\mathrm{d}t \right\rvert\\
  \leq &\frac{1}{\Gamma(\rho)} \sup_{u\in [a,b]}\sup_{f\in \mathfrak{F}} \int_{a}^{u}\lvert (u-t)^{\rho-1} \rvert \lvert f(t) \rvert ~\mathrm{d}t\\
  \leq & \frac{\mathcal{M}}{\Gamma(\rho)}\int_{a}^{u} \lvert (u-t)^{\rho-1}\rvert ~\mathrm{d}t\\
  \leq & \frac{\mathcal{M}}{\Gamma(\rho)} \frac{(u-a)^{\rho}}{\rho}.
\end{align*}
Hence,
$$H_d\left(~_{a}\mathfrak{J}^{\rho}F(u), \{0\}\right) \leq \frac{\mathcal{M}}{\Gamma(\rho)}\frac{(b-a)^{\rho}}{\rho}, \text{ for all } u\in [a,b].$$
This completes the proof.
\end{proof}

\begin{theorem}\label{conti}
   Let $F:[a,b]\rightrightarrows\mathbb{R}$ be a set-valued mapping with nonempty compact values, Borel measurable and bounded on $[a,b]$.
Let $\rho>0$. Then, for each $u\in[a,b]$, the set ${}_{a}\mathfrak{J}^{\rho}F(u)$ is nonempty and compact,  and the mapping
$$u \longmapsto {}_{a}\mathfrak{J}^{\rho}F(u)$$
 is continuous on $[a,b]$ with respect to $H_d$.
\end{theorem}
\begin{proof}
Since $F$ is Borel-measurable, bounded and has nonempty compact values, Theorem \ref{THnonempty} guarantees that, for each $u\in[a,b]$, the set of
integrable selections of $F$ is nonempty and therefore ${}_{a}\mathfrak{J}^{\rho}F(u)\neq\emptyset$.
Let $\mathfrak{F}$ denote the family of all integrable selections of $F$. Since $F$ is bounded on $[a,b]$, it is also integrably bounded.
In particular, we may choose a nonnegative constant function
$h\in L^{1}([a,b])$ such that
$$\sup\{|x|:x\in F(t)\}\le h(t)
\quad\text{for all }t\in[a,b].$$
In particular, for each $f\in\mathfrak{F}$ we have
\begin{equation}\label{EQ1}
|f(t)|\le h(t)\quad\text{for a.e. }t\in[a,b].
\end{equation}
For $f\in\mathfrak{F}$ and $u\in[a,b]$, write
$$I_f(u):=\frac{1}{\Gamma(\rho)}\int_{a}^{u}(u-t)^{\rho-1}f(t)dt,$$
so that
${}_{a}\mathfrak{J}^{\rho}F(u)=\{I_f(u):f\in\mathfrak{F}\}.$
For any $f\in\mathfrak{F}$ and $u\in[a,b]$, by using \eqref{EQ1} we have
\begin{align*}
    |I_f(u)|
\le
\frac{1}{\Gamma(\rho)}\int_{a}^{u}(u-t)^{\rho-1}|f(t)|dt
\le
\frac{1}{\Gamma(\rho)}\int_{a}^{u}(u-t)^{\rho-1}h(t)dt.
\end{align*}

Since $\rho>0$ and
$h\in L^{1}([a,b])$, the right-hand side is finite. 
Thus there exists $M>0$ such that $ |I_f(u)|\le M
\quad\text{for all }u\in[a,b]\text{ and all }f\in\mathfrak{F},$
and therefore ${}_{a}\mathfrak{J}^{\rho}F(u)$ is bounded for each $u$. At $u=a$, the set equals $\{0\}$. For $u>a$, to see that
${}_{a}\mathfrak{J}^{\rho}F(u)$ is closed, let
$\{I_{f_n}(u)\}_{n\in\mathbb{N}}\subset {}_{a}\mathfrak{J}^{\rho}F(u)$
be a sequence converging to some $y\in\mathbb{R}$. By \eqref{EQ1},
\begin{equation}\label{EQ2}
|f_n(t)|\le h(t)
\qquad\text{for a.e. }t\in[a,b]\text{ and for all }n\in\mathbb{N}.
\end{equation}
The mapping
$$
t\longmapsto \frac{(u-t)^{\rho-1}}{\Gamma(\rho)}F(t),
\qquad t\in[a,u),
$$
with value $\{0\}$ at $t=u$, has nonempty compact values and is
integrably bounded by $(u-t)^{\rho-1}h(t)/\Gamma(\rho)$. Indeed, since $h$ is constant,
$$
\frac{1}{\Gamma(\rho)}
\int_a^u (u-t)^{\rho-1}h(t)\,dt
=
\frac{h(a)(u-a)^\rho}{\Gamma(\rho+1)}
<\infty.
$$
Its Aumann integral over $[a,u]$ is compact by
\cite[Theorem~4]{aumann1965integrals}. Each $I_{f_n}(u)$ belongs to
this integral, so its limit $y$ also belongs to it.
Dividing the corresponding selection by the positive kernel on
$[a,u)$ and extending it to $[u,b]$ using any fixed member of
$\mathfrak{F}$ gives a measurable selection $f$ of $F$.
Since $|f(t)|\le h(t)$ almost everywhere, $f\in\mathfrak{F}$, and
$y=I_f(u)\in{}_{a}\mathfrak{J}^{\rho}F(u)$.
Therefore ${}_{a}\mathfrak{J}^{\rho}F(u)$ is closed.
Being closed and bounded in $\mathbb{R}$, it is compact.
Together with nonemptiness from Theorem \ref{THnonempty}, this shows that
${}_{a}\mathfrak{J}^{\rho}F(u)$ is a nonempty compact set for every $u\in[a,b]$.

Let $u,v\in[a,b]$ and $f\in\mathfrak{F}$ be arbitrary. Without loss of generality assume $u\le v$. Then
$$
\begin{aligned}
I_f(v)-I_f(u)
=
\frac{1}{\Gamma(\rho)}
\Bigg(
\int_{a}^{v}(v-t)^{\rho-1}f(t)dt
-
\int_{a}^{u}(u-t)^{\rho-1}f(t)dt
\Bigg) \\
=
\frac{1}{\Gamma(\rho)}
\Bigg(
\int_{a}^{u}\big[(v-t)^{\rho-1}-(u-t)^{\rho-1}\big]f(t)dt
+
\int_{u}^{v}(v-t)^{\rho-1}f(t)dt
\Bigg).
\end{aligned}
$$
From \eqref{EQ1}, we have
$\big|I_f(v)-I_f(u)\big| \le \Phi(u,v),$
where
$$
\Phi(u,v) := \frac{1}{\Gamma(\rho)}
\Bigg( \int_{a}^{u}\big|(v-t)^{\rho-1}-(u-t)^{\rho-1}\big|h(t)dt + \int_{u}^{v}(v-t)^{\rho-1}h(t)dt \Bigg).
$$
We now show that $\Phi(u,v)\to 0$ as $v\to u$. For the first integral,
note that for fixed $u\in[a,b]$ and for a.e.~$t\in[a,u)$,
$$ (v-t)^{\rho-1}\to (u-t)^{\rho-1}
\quad\text{as }v\to u,\quad v\ge u, $$
and hence
$$ \big|(v-t)^{\rho-1}-(u-t)^{\rho-1}\big|
\to 0 \quad\text{for a.e. }t\in[a,u]. $$
Moreover, since $\rho>0$, we have
$$ \big|(v-t)^{\rho-1}-(u-t)^{\rho-1}\big|
\le (v-t)^{\rho-1}+(u-t)^{\rho-1}. $$
For $u\le v\le b$ and $t<u$, this gives
$$
\big|(v-t)^{\rho-1}-(u-t)^{\rho-1}\big|h(t)
\le
2h(t)\big[(u-t)^{\rho-1}+(b-a)^{\rho-1}\big].
$$
Since $h$ is constant, the right-hand side belongs to
$L^1([a,u])$ and is independent of $v$.
Hence, by the Dominated Convergence Theorem,
$$ \int_{a}^{u}\big|(v-t)^{\rho-1}-(u-t)^{\rho-1}\big|\,h(t)dt \longrightarrow 0 \quad\text{as }v\to u,\quad v\ge u. $$
For the second integral, observe that for $v>u$,
$$
0 \le \int_u^v (v-t)^{\rho-1} h(t)\,dt
=
\frac{h(a)}{\rho}(v-u)^\rho.
$$
Since $\rho>0$, the right-hand side tends to $0$ as $v\to u$.
For $u>v$, define $\Phi(u,v):=\Phi(v,u)$. Since $h$ is constant,
direct integration also gives
$$
\Phi(u,v)
\le
\frac{h(a)}{\Gamma(\rho+1)}
\left(
\big|(v-a)^\rho-(u-a)^\rho\big|+2|v-u|^\rho
\right),
\qquad u,v\in[a,b].
$$
Consequently,
$ \Phi(u,v)\to 0 \quad\text{whenever }v\to u.$
Now, fix $u,v\in[a,b]$ and take any $y\in {}_{a}\mathfrak{J}^{\rho}F(u)$. Then
$y=I_f(u)$ for some $f\in\mathfrak{F}$. Let $z:=I_f(v)\in
{}_{a}\mathfrak{J}^{\rho}F(v)$. By the above,
$$ |y-z|=\big|I_f(u)-I_f(v)\big|\le\Phi(u,v), $$
and hence
$$\sup_{y\in {}_{a}\mathfrak{J}^{\rho}F(u)} \inf_{z\in {}_{a}\mathfrak{J}^{\rho}F(v)}|y-z| \le\Phi(u,v). $$
Exchanging the roles of $u$ and $v$ yields
$$ \sup_{z\in {}_{a}\mathfrak{J}^{\rho}F(v)} \inf_{y\in {}_{a}\mathfrak{J}^{\rho}F(u)}|z-y| \le\Phi(u,v). $$
By the definition of the Hausdorff distance, we therefore have
$$ H_d\big({}_{a}\mathfrak{J}^{\rho}F(u),{}_{a}\mathfrak{J}^{\rho}F(v)\big) \le\Phi(u,v)
\quad\text{for all }u,v\in[a,b]. $$
Since $\Phi(u,v)\to 0$ as $v\to u$, this shows that the mapping
$$u\longmapsto {}_{a}\mathfrak{J}^{\rho}F(u) $$
is continuous on $[a,b]$ with respect to the Hausdorff distance $H_d$. This completes the proof.
\end{proof}
\begin{theorem}\label{BV}
Let $F:[a,b]\rightrightarrows\mathbb{R}$ be a set-valued mapping with nonempty compact convex values, Borel measurable and integrably bounded on $[a,b]$.
Assume that $F$ is of bounded variation on $[a,b]$ with respect to $H_d$ and that $\rho>1$. Then, for each $u\in[a,b]$, the set ${}_{a}\mathfrak{J}^{\rho}F(u)$ is nonempty, compact and convex, and the mapping
$$ u \longmapsto {}_{a}\mathfrak{J}^{\rho}F(u) $$ is of bounded variation on $[a,b]$ with respect to the $H_d$.
\end{theorem}

\begin{proof}
Since $F$ has compact values and is of bounded variation with respect to $H_d$, it is bounded on $[a,b]$.
Therefore, as $F$ is Borel-measurable, integrably bounded and has nonempty compact values, Theorem~\ref{conti} applied with
$\rho>1$ yields that, for each $u\in[a,b]$, the set ${}_{a}\mathfrak{J}^{\rho}F(u)$ is nonempty and compact and that the mapping $u\mapsto {}_{a}\mathfrak{J}^{\rho}F(u)$ is continuous on $[a,b]$ with respect to $H_d$. Moreover, by Theorem \ref{convex}, since $F$ has convex values, each value ${}_{a}\mathfrak{J}^{\rho}F(u)$ is convex. Thus it remains to show that $u\mapsto {}_{a}\mathfrak{J}^{\rho}F(u)$ has bounded variation on $[a,b]$ with respect to $H_d$.
For each $u\in[a,b]$, the set $F(u)$ is a nonempty compact convex subset of $\mathbb{R}$, hence a closed interval. Define
$$\underline f(u):=\inf F(u),\qquad \overline f(u):=\sup F(u), $$
so that
$$ F(u)=[\underline f(u),\overline f(u)], \qquad u\in[a,b].
$$
Let $\mathcal{P}=\{a=u_0<u_1<\dots<u_n=b\}$ be a partition of $[a,b]$. Then for every $i=1,\dots,n$,
\begin{align*}
    H_d\big(F(u_i),F(u_{i-1})\big) & =
H_d\big([\underline f(u_i),\overline f(u_i)],
        [\underline f(u_{i-1}),\overline f(u_{i-1})]\big) \\&  = \max\Big\{|\underline f(u_i)-\underline f(u_{i-1})|,
          |\overline f(u_i)-\overline f(u_{i-1})|\Big\}.
\end{align*}
Consequently,
$$ \sum_{i=1}^{n}|\underline f(u_i)-\underline f(u_{i-1})|
\le \sum_{i=1}^{n}H_d\big(F(u_i),F(u_{i-1})\big), $$
and similarly for $\overline f$. Taking the supremum over all partitions
$\mathcal{P}$, we obtain
$$ V(\underline f,[a,b])\le V(F,[a,b]),\qquad V(\overline f,[a,b])\le V(F,[a,b])<\infty. $$
In particular, $\underline f$ and $\overline f$ are functions of bounded variation on $[a,b]$ and therefore are bounded on $[a,b]$. For $u\in[a,b]$ define
$$ \big({}_{a}\mathfrak{J}^{\rho}\underline f\big)(u)
:= \frac{1}{\Gamma(\rho)}\int_{a}^{u}(u-t)^{\rho-1}\underline f(t)dt,$$ and 
$$
\big({}_{a}\mathfrak{J}^{\rho}\overline f\big)(u)
:= \frac{1}{\Gamma(\rho)}\int_{a}^{u}(u-t)^{\rho-1}\overline f(t)dt. $$
Since $\underline f$ and $\overline f$ are bounded and integrable, the functions ${}_{a}\mathfrak{J}^{\rho}\underline f$ and
${}_{a}\mathfrak{J}^{\rho}\overline f$ are well defined. Now, we claim that both of them are Lipschitz , hence of bounded variation, on $[a,b]$.
Let $g$ be either $\underline f$ or $\overline f$, and set
$$
I_g(u):=\frac{1}{\Gamma(\rho)}\int_{a}^{u}(u-t)^{\rho-1}g(t)dt.
$$
Since $g$ is bounded on $[a,b]$, there exists $M>0$ such that $|g(t)|\le M$
for all $t\in[a,b]$. For $u\in(a,b]$, differentiating under the integral
sign (since $\rho>1$ and the kernel
$(u-t)^{\rho-2}$ is integrable on $[a,u]$) gives
$$I'_g(u)
=
\frac{\rho-1}{\Gamma(\rho)}
\int_{a}^{u}(u-t)^{\rho-2}g(t)dt.$$
Therefore,
\begin{align*}
    |I'_g(u)|
&\le
\frac{\rho-1}{\Gamma(\rho)}
\int_{a}^{u}(u-t)^{\rho-2}|g(t)|dt
\\& \le
\frac{\rho-1}{\Gamma(\rho)}M\int_{a}^{u}(u-t)^{\rho-2}dt \\& =\frac{M}{\Gamma(\rho)}(u-a)^{\rho-1}
\\& \le
\frac{M}{\Gamma(\rho)}(b-a)^{\rho-1}.
\end{align*}
Hence $I'_g$ is bounded on $(a,b]$ and extends continuously to $[a,b]$,
so $I_g$ is Lipschitz on $[a,b]$ with Lipschitz constant
$$
L_0:=\frac{M}{\Gamma(\rho)}(b-a)^{\rho-1}.
$$
Thus both ${}_{a}\mathfrak{J}^{\rho}\underline f$ and
${}_{a}\mathfrak{J}^{\rho}\overline f$ are Lipschitz (in fact with the same
constant $L_0$), and therefore they are functions of bounded variation on
$[a,b]$.

We now show that, for each $u\in[a,b]$,
\begin{equation}\label{EQ3}
{}_{a}\mathfrak{J}^{\rho}F(u)
=
\Big[
\big({}_{a}\mathfrak{J}^{\rho}\underline f\big)(u),
\big({}_{a}\mathfrak{J}^{\rho}\overline f\big)(u)
\Big].
\end{equation}
Let $f\in\mathfrak{F}$ be an integrable selection of $F$, so that
$\underline f(t)\le f(t)\le\overline f(t)$ for all $t\in[a,b]$. Since the
kernel $(u-t)^{\rho-1}/\Gamma(\rho)$ is non negative for $t\in[a,u]$, we obtain
$$
\big({}_{a}\mathfrak{J}^{\rho}\underline f\big)(u)
\le I_f(u)\le
\big({}_{a}\mathfrak{J}^{\rho}\overline f\big)(u),
$$ 
which shows that every element of ${}_{a}\mathfrak{J}^{\rho}F(u)$ belongs
to the interval on the right-hand side of \eqref{EQ3}. On the other hand, $\underline f$ and $\overline f$ themselves are selections of
$F$, so the corresponding integrals
${}_{a}\mathfrak{J}^{\rho}\underline f(u)$ and
${}_{a}\mathfrak{J}^{\rho}\overline f(u)$ belong to
${}_{a}\mathfrak{J}^{\rho}F(u)$. Since ${}_{a}\mathfrak{J}^{\rho}F(u)$ is convex by Theorem~\ref{convex} and contains both endpoints, it contains the interval in \eqref{EQ3}. Thus, equality follows.
Thus, for every $u\in[a,b]$,
\begin{equation*}
    {}_{a}\mathfrak{J}^{\rho}F(u)
=
\big[ A(u),B(u)\big],
\qquad \text{where}\quad
A(u):={}_{a}\mathfrak{J}^{\rho}\underline f(u),\quad
B(u):={}_{a}\mathfrak{J}^{\rho}\overline f(u).
\end{equation*}
In particular, ${}_{a}\mathfrak{J}^{\rho}F(u)$ is compact and convex. Let $\mathcal{P}=\{a=u_0<\dots<u_n=b\}$ be a partition of $[a,b]$. Then for each $i=1,\dots,n$,
\begin{align*}
    H_d\big({}_{a}\mathfrak{J}^{\rho}F(u_i),{}_{a}\mathfrak{J}^{\rho}F(u_{i-1})\big)
&= H_d\big([A(u_i),B(u_i)],[A(u_{i-1}),B(u_{i-1})]\big) \\& =\max\big\{|A(u_i)-A(u_{i-1})|,
         |B(u_i)-B(u_{i-1})|\big\}.
\end{align*}
Hence
$$\sum_{i=1}^{n}
H_d\big({}_{a}\mathfrak{J}^{\rho}F(u_i),{}_{a}\mathfrak{J}^{\rho}F(u_{i-1})\big)
\le
\sum_{i=1}^{n}\Big(|A(u_i)-A(u_{i-1})|
                   +|B(u_i)-B(u_{i-1})|\Big).$$
Taking the supremum over all partitions $\mathcal{P}$, we obtain
$$
V\big({}_{a}\mathfrak{J}^{\rho}F,[a,b]\big)
\le
V(A,[a,b])+V(B,[a,b])<\infty,
$$
since $A$ and $B$ are of bounded variation. Thus
$u\mapsto{}_{a}\mathfrak{J}^{\rho}F(u)$ is of bounded variation on $[a,b]$ with respect to $H_d$. This completes the proof.
\end{proof}

\begin{theorem} \label{LP}
Let $F:[a,b]\rightrightarrows\mathbb{R}$ be a set-valued mapping with nonempty compact values, Borel measurable and integrably bounded on $[a,b]$.
Assume that $F$ is Lipschitz on $[a,b]$ with respect to $H_d$ and that $\rho>1$. Then, for each $u\in[a,b]$, the set ${}_{a}\mathfrak{J}^{\rho}F(u)$ is nonempty and compact, and the mapping
$$
u \longmapsto {}_{a}\mathfrak{J}^{\rho}F(u)
$$
is Lipschitz on $[a,b]$ with respect to $H_d$.
\end{theorem}
    \begin{proof}
Since $F$ has compact values and is Lipschitz with respect to $H_d$, it is bounded on $[a,b]$. Hence, by Theorem~\ref{conti},
${}_{a}\mathfrak{J}^{\rho}F(u)$ is nonempty and compact for every $u\in[a,b]$. Thus, it remains to show that the mapping $u\mapsto{}_{a}\mathfrak{J}^{\rho}F(u)$ is Lipschitz with respect to $H_d$.     
Let $L_F\ge0$ be a Lipschitz constant of $F$ with respect to $H_d$, that is
$$H_d\big(F(u),F(v)\big)\le L_F|u-v|
\quad\text{for all }u,v\in[a,b].$$

To prove the Lipschitz continuity of ${}_{a}\mathfrak{J}^{\rho}F$, we first  establish a uniform bound on the values of $F$. Fix $u_{0}\in[a,b]$ and  choose any $x_{0}\in F(u_{0})$. Since $F(u_{0})$ is compact, it is bounded,  and therefore 
$$
R_{0}:=\sup\{\,|x|:\,x\in F(u_{0})\,\}<\infty.
$$
For arbitrary $u\in[a,b]$ and $x\in F(u)$, by the definition
of the Hausdorff distance there exists $y\in F(u_0)$ such that
$$|x-y|\le H_d\big(F(u),F(u_0)\big).$$
Because the Lipschitz continuity of $F$, we have
$$|x-y|\le H_d\big(F(u),F(u_0)\big)\le L_F|u-u_0|\le L_F(b-a).$$
Since $y\in F(u_{0})$, we have $|y|\le R_{0}$, and hence
$$
|x|\le |y|+|x-y|\le R_{0}+L_{F}(b-a)=:M.
$$
Thus,
$$
\sup\{\,|x|:\,x\in F(u)\,\}\le M
\qquad\text{for all }u\in[a,b].
$$
Consequently, every integrable selection $f$ of $F$ satisfies
\begin{equation}\label{EQ4}
|f(t)|\le M
\qquad\text{for a.e. }t\in[a,b].
\end{equation}

For such a selection $f$, define
$$
I_f(u):=\frac{1}{\Gamma(\rho)}\int_{a}^{u}(u-t)^{\rho-1}f(t)dt,
\qquad u\in[a,b].
$$
As in the proof of the above theorem,  since $\rho>1$ and $f$ is bounded, we may differentiate under the integral sign to obtain, for
$u\in(a,b]$,
$$
I_f'(u)
=
\frac{\rho-1}{\Gamma(\rho)}
\int_{a}^{u}(u-t)^{\rho-2}f(t)dt.
$$

Using \eqref{EQ4}
$$|I_f'(u)|
\le
\frac{\rho-1}{\Gamma(\rho)}M\int_{a}^{u}(u-t)^{\rho-2}dt
=
\frac{M}{\Gamma(\rho)}(u-a)^{\rho-1}
\le
\frac{M}{\Gamma(\rho)}(b-a)^{\rho-1}
=:L_0.$$
Hence $I_f$ is Lipschitz on $[a,b]$ with Lipschitz constant at most $L_0$, and the constant $L_0$ is independent of the particular selection $f$.
Now let $u,v\in[a,b]$ and $f\in\mathfrak{F}$ be an integrable selection of
$F$. Then
$$
|I_f(u)-I_f(v)|\le L_0|u-v|.
$$
If $y\in {}_{a}\mathfrak{J}^{\rho}F(u)$, then $y=I_f(u)$ for some
$f\in\mathfrak{F}$, and taking $z:=I_f(v)\in{}_{a}\mathfrak{J}^{\rho}F(v)$
we obtain
$|y-z|\le L_0|u-v|.$
Therefore
$$
\sup_{y\in {}_{a}\mathfrak{J}^{\rho}F(u)}
\inf_{z\in {}_{a}\mathfrak{J}^{\rho}F(v)}|y-z|
\le L_0|u-v|.
$$
By symmetry (interchanging $u$ and $v$), we also have
$$
\sup_{z\in {}_{a}\mathfrak{J}^{\rho}F(v)}
\inf_{y\in {}_{a}\mathfrak{J}^{\rho}F(u)}|z-y|
\le L_0|u-v|.
$$
Hence, by the definition of the Hausdorff distance,
$$
H_d\big({}_{a}\mathfrak{J}^{\rho}F(u),
        {}_{a}\mathfrak{J}^{\rho}F(v)\big)
\le L_0|u-v|
\quad\text{for all }u,v\in[a,b],
$$
which shows that $u\mapsto{}_{a}\mathfrak{J}^{\rho}F(u)$ is Lipschitz on $[a,b]$ with respect to $H_d$. This completes the proof.
\end{proof}
%\subsection*{Existence of regular selections:}
\begin{theorem} \label{Maintheorem}
    Let $F:[a,b]\rightrightarrows\mathbb{R}$ be a set-valued mapping with nonempty
compact values, Borel measurable and integrably bounded on $[a,b]$, and let
$\rho>1$. Let ${}_{a}\mathfrak{J}^{\rho}F$ be a set-valued Riemann-Liouville fractional integral of $F$. Then the following hold:
    \begin{itemize}
        \item[(a)] If $F$ has compact convex values and is of bounded variation on $[a,b]$ with respect to $H_d$, then the mapping
$$ u \longmapsto {}_{a}\mathfrak{J}^{\rho}F(u)$$ admits a continuous selection which is itself of bounded variation.
        \item[(b)] If $F$ is Lipschitz on $[a,b]$ with respect to $H_d$, then the mapping
$$ u \longmapsto {}_{a}\mathfrak{J}^{\rho}F(u)$$ admits a Lipschitz selection.
    \end{itemize}
\end{theorem}
\begin{proof} Now, Theorems \ref{conti}, \ref{BV}, \ref{LP} and Theorem 3 of
Belov and Chistyakov \cite{SAVV} will be used in the proof of this theorem.\\
Let $\rho>1$ and let $F:[a,b]\rightrightarrows\mathbb{R}$ satisfy the assumptions of Theorem \ref{Maintheorem}.  Define the set-valued mapping
$$G(u):={}_{a}\mathfrak{J}^{\rho}F(u), \qquad u\in[a,b].$$

Under either assumption $(a)$ or $(b)$, $F$ has compact values and is of bounded variation or Lipschitz continuous with respect to $H_d$,
and hence is bounded on $[a,b]$. Therefore, Theorem~\ref{conti} implies that $G(u)$ is nonempty and compact for every $u\in[a,b]$. 
\begin{itemize}
    \item[(a)] From Theorems \ref{conti} and \ref{BV}, $G$ is continuous and of bounded variation on $[a,b]$ with respect to Hausdorff distance $H_d$. Theorem 3 of Belov and Chistyakov \cite{SAVV} guarantees the existence of a regular selection $g:[a,b]\to\mathbb{R}$ of $G$, that is,
$$g(u)\in G(u)={}_{a}\mathfrak{J}^{\rho}F(u)
\quad\text{for all }u\in[a,b],$$
and $g$ is continuous and of bounded variation on $[a,b]$ with 
$$V(g,[a,b]) \le V(G,[a,b]).$$

    \item[(b)] From Theorem \ref{LP}, $G$ is Lipschitz on $[a,b]$ with respect to the Hausdorff metric $H_d$. Theorem 3 of Belov and Chistyakov \cite{SAVV} yields a Lipschitz selection $g:[a,b]\to\mathbb{R}$ of $G$, i.e.,
$$g(u)\in G(u)={}_{a}\mathfrak{J}^{\rho}F(u)
\quad\text{for all }u\in[a,b],$$
 with $\operatorname{Lip}(g)\ \le\ \operatorname{Lip}(G).$
\end{itemize}

This completes the proof.
\end{proof}
\begin{remark}
The convexity of the values of $F$ is only used in the proof of the bounded variation property of the Riemann–Liouville fractional integral in Theorem \ref{BV}. The subsequent existence of regular selections via the result of Belov and Chistyakov \cite{SAVV} does not require convexity.
\end{remark}

\begin{theorem}
Let $F:[a,b]\rightrightarrows\mathbb{R}$ be a set-valued mapping
with nonempty compact values, Borel measurable and bounded on
$[a,b]$, and let $\rho>0$. For each $u\in[a,b]$, let
$$
G(u):={}_{a}\mathfrak{J}^{\rho}F(u).
$$
By Theorem~\ref{conti}, $G(u)$ is a nonempty compact subset of
$\mathbb{R}$ for each $u\in[a,b]$. Define the extremal selections
$$
g_{-}(u):=\min G(u),\qquad
g_{+}(u):=\max G(u),\qquad u\in[a,b].
$$
Then the following assertions hold:
\begin{itemize}
    \item[(a)] Under the assumptions of Theorem \ref{conti},  $G$ is continuous on $[a,b]$, and both $g_{-}$ and $g_{+}$ are continuous  selections of $G$ on $[a,b]$.
    
    \item[(b)] Under the assumptions of Theorem \ref{BV} with $\rho>1$,
$G$ is of bounded variation on $[a,b]$, and both $g_{-}$ and $g_{+}$ are of 
bounded variation on $[a,b]$. Moreover,
$$ V\big(g_{-},[a,b]\big)\le V\big(G,[a,b]\big), \qquad V\big(g_{+},[a,b]\big)\le V\big(G,[a,b]\big). $$
\item[(c)] Under the assumptions of Theorem \ref{LP} with $\rho>1$, $G$ is 
Lipschitz on $[a,b]$, and both $g_{-}$ and $g_{+}$ are Lipschitz on $[a,b]$. 
In addition,
$$ \operatorname{Lip}(g_{-})\le \operatorname{Lip}(G), \qquad \operatorname{Lip}(g_{+})\le\operatorname{Lip}(G).
$$
\end{itemize}
\end{theorem}
    \begin{proof}
By Theorem~\ref{conti}, each $G(u)$ is a nonempty compact subset of $\mathbb{R}$, hence $\min G(u)$ and $\max G(u)$ are well defined and yield selections $g_{-}$ and $g_{+}$. Indeed, for all $u,v\in[a,b]$, we have
$ |g_{\pm}(u)-g_{\pm}(v)|
\le H_d(G(u),G(v)).$ The continuity, bounded variation, and Lipschitz regularity of $g_{-}$ and $g_{+}$ follow directly from the corresponding regularity of $G$ (from Theorems \ref{conti}, \ref{BV} and \ref{LP}) and standard properties of the Hausdorff distance on $\mathbb{R}$.
\end{proof}
\begin{remark}
Similarly, one may define the set-valued right-sided Riemann-Liouville fractional integral of $F$ by
$$
\mathfrak{J}_{b}^{\rho}F(u)
=
\left\{
\frac{1}{\Gamma(\rho)}
\int_{u}^{b}(t-u)^{\rho-1}f(t)\,dt
:\ f\in\mathfrak{F}
\right\},~~ u\in[a,b].
$$
 All the results proved above 
admit analogous right-sided versions.  The proofs follow in the same way by replacing $[a,u]$ with $[u,b]$ and $(u-t)^{\rho-1}$ with $(t-u)^{\rho-1}$.
\end{remark}

\section*{\textbf{Application}}
One of the important applications of research presented in this paper is in one of the area of differential inclusion. For more detailed theory of fractional differential inclusions, we refer to \cite{fcaa2,fcaa1}. In order to prove existence and uniqueness, one can directly utilize the selection theorem, i.e., Theorem \ref{Maintheorem}.   \\
 Let $\alpha\in(1,2)$ and consider the Caputo fractional
differential inclusion
\begin{equation}\label{CPEQN}
{}^{C}D^{\alpha}u(t)\in F(t,u(t))
\quad\text{for a.e. }t\in[t_{0},T],
\qquad
u(t_{0})=u_{0},\quad u'(t_{0})=u_{1},
\end{equation}
where ${}^{C}D^{\alpha}$ denotes the Caputo derivative of order $\alpha$ and
$F:[t_{0},T]\times\mathbb{R}\rightrightarrows\mathbb{R}$ is a set-valued
mapping with nonempty compact values.

 The existence of a solution usually requires that $F(t, u)$ be measurable in $t$, upper hemicontinuous function of $u$, and $F(t, u)$ a closed, convex set for all $(t,u)$. The uniqueness requires the Lipschitz condition on $F$, i.e. there exists $L>0$ such that for every $t\in[t_0,T]$ and every $u_1,u_2 \in \mathbb{R},$
 Hausdorff distance
 $$H_d(F(t, u_1), F(t, u_2)) \leq L \|u_1 - u_2\|.$$ For more details, we refer to \cite{incl}. The existence of a solution of Cauchy problem \eqref{CPEQN} is equivalent to looking for fixed points of the corresponding integral form. We look for a solution in the set $$_{a}\mathfrak{J}^{\rho}F(t)=\left\{\frac{1}{\Gamma(\rho)}\int_{a}^{t}(t-s)^{\rho-1}f(s)~\mathrm{d}s : f\in \mathfrak{F}\right\}.$$ 
 It is well known that problem \eqref{CPEQN} is equivalent to
the integral inclusion
\begin{equation}\label{INGEQN}
u(t)\in u_{0}+u_{1}(t-t_{0})
+\frac{1}{\Gamma(\alpha)}
\int_{t_{0}}^{t}(t-s)^{\alpha-1}v(s)\,ds,
\qquad v(s)\in F(s,u(s))\ \text{a.e.}
\end{equation}
for $t\in[t_{0},T]$.

Setting $\rho=\alpha>1$, the integral term in \eqref{INGEQN}
coincides with the Riemann-Liouville fractional integral of order $\rho$
applied to a selection of the mapping $s\mapsto F(s,u(s))$. By using the results proved, it is evident that a selection exists (Theorem \ref{Maintheorem}) and hence a solution exists. Hence, we can directly apply the results on the problem of differential inclusions of fractional order to prove the existence and uniqueness of the solution.

One can also consider the following fractional differential inclusion with Riemann-Liouville derivative of order $\alpha \in (1,2],$
$${}^{RL}D_{t_0}^{\alpha }u(t) \in F(t,u(t))$$ a.e. for $t\in[t_{0},T],$ with initial conditions, $$\lim _{t\rightarrow {t_0}^{+}}(t-t_0)^{1-\alpha }u(t)=c_{0}, \lim _{t\rightarrow {t_0}^{+}}(t-t_0)^{2-\alpha }u^{\prime }(t)=c_{1}$$ which has the integral form: $$u(t) \in c_{0}\frac{(t-t_0)^{\alpha -1}}{\Gamma (\alpha )}+c_{1}\frac{(t-t_0)^{\alpha -2}}{\Gamma (\alpha -1)}+\frac{1}{\Gamma (\alpha )}\int _{t_0}^{t}(t-s)^{\alpha -1}v(s)\,ds,$$ with $$v(s) \in F(s,u(s))$$ a.e. for $t\in[t_{0},T]$. Using a similar argument as above, the existence of a solution is ensured using the selection theorem  (Theorem \ref{Maintheorem}).

\section{Conclusion and future remarks} In this paper, we introduced the concept of set-valued Riemann-Liouville fractional integral and established several fundamental properties of this operator, including convexity, boundedness, continuity with respect to the Hausdorff metric, bounded variation, and Lipschitz regularity under suitable assumptions. We also proved the existence of regular selections for the set-valued Riemann-Liouville fractional integral and identified extremal selections in the scalar setting.
We do a systematic study of regularity preservation under the fractional integral operator, which is a non-trivial property for multifunctions.

We have shown the application in fractional differential inclusions. The Cauchy-type problems involving set-valued fractional differential equations can be solved or approximated using regular selections of the fractional integral. The paradigm is applicable to modeling with uncertainty or imprecision in the field such as economics, diffusion with memory, robust control, viscoelasticity, interval-valued data, and even reinforcement learning with set-valued policies. Alos, we can look for some connections with fractal geometrey.
The computation of fractal dimensions has long been a central and fascinating topic in fractal theory. In the past few years, the study of fractal dimension (such as box dimension and the Hausdorff dimension) of the graph of the integral of a function is getting the attention of the researchers. For example, Liang \cite{L1} introduced the box dimension of the graph of the Riemann-Liouville fractional integral of a continuous function of bounded variation on a closed bounded interval. Subsequently, several significant developments in this direction have been reported; see, for instance, \cite{SS,SSS,SS1,L2,MV,SV,SV1,BY,BY1}. Motivated by these advances, it would be natural and interesting to investigate analogous dimensional properties for set-valued Riemann–Liouville fractional integrals, including the fractal dimensions of their graphs under suitable regularity assumptions on the underlying set-valued mappings.

This work presents a foundational contribution that cleanly bridges fractional calculus and set-valued analysis. It presents clear theoretical results on regularity transfer and selection existence that open the door to applications in differential inclusions under uncertainty.
\subsection*{Acknowledgements} The fruitful discussion during Fractals: GAANA 2025 organized by IIIT Prayagraj is acknowledged.

\bibliographystyle{amsplain}

\begin{thebibliography}{10}
\bibitem{abbas2} Abbas, S.:  Existence of solutions to fractional order ordinary and delay differential equations and applications. Electron. J. Differential Equations {\bf 2011}(09), 1-11 (2011)

\bibitem{lup1}Abbas,  U.,   Lupulescu, V.,  O’Regan, D.,  Awais, Y.:  Neutral set differential equations. Czech Math. J. {\bf 65}, 593–615 (2015) 

\bibitem{azam} Abbas, A.,  Azam, A.,  Mehmood, N.,   Ali, F.: Existence of solutions to fractional differential inclusions with non-separated and multipoint boundary constraints. Fractal and Fractional. {\bf 9}(9), (2025)
\bibitem{fcaa2}  Abbas, S., Benchohra, M.,  Lazreg, J.E.,  Nieto, J.J.,  Zhou, Y.: Fractional Differential Equations and Inclusions: Classical and Advanced Topics. World Scientific (2023)

\bibitem{mokh1} Ahmad, S., Kirane, M.:  On a fractional‐order mathematical model to assess the impact of diabetes and its associated complications in the United Arab Emirates. Math. Methods Appl. Sci. {\bf47}(8), 6892-6902 (2024)

\bibitem{mokh2}  Alsaedi, A.,  Ahmad, B.,  Kirane, M: A survey of useful inequalities in fractional calculus. Fract. Calc. Appl. Anal. {\bf 20}(3),  574–594 (2017). https://doi.org/10.1515/fca-2017-0031

\bibitem{jehad1}  Alzabut, J.  Houas, M.,  Abbas, M.I.:  Application of fractional quantum calculus on coupled hybrid differential systems within the sequential Caputo fractional  q-derivatives. Demonstr. Math. {\bf 56}(1), 20220205 (2023)

\bibitem{aumann1965integrals} Aumann, R.J.: Integrals of set-valued functions. J. Math. Anal. Appl. {\bf12}(1), 1-12 (1965)
\bibitem{art}  Artstein, Z.: Set-valued measures. Trans. Amer. Math. Soc. {\bf 165},  103–125 (1972) 
\bibitem{incl} Aubin, J.P., Cellina, A.:  Differential Inclusions: Set-Valued Maps and Viability Theory. Grundlehren Math. Wiss. 264, Springer Verlag, Berlin (1984)

\bibitem{aubin1} Aubin, J.P.,   Frankowska, H.:  Set-Valued Analysis. Birkhauser (1990)

\bibitem{SAVV}  Belov, S.A.,  Chistyakov, V.V.:  A selection principle for mappings of bounded variation. J. Math. Anal. Appl. {\bf249}, 351-366 (2000)
\bibitem{HBTS} Budak, H., Tunç, T., Sarikaya, M.:  Fractional Hermite-Hadamard-type inequalities for interval-valued functions. Proc. Amer. Math. Soc. {\bf148}, 705–718 (2020)
\bibitem{GB} Beer, G.:  Metric spaces on which continuous functions are uniformly continuous and Hausdorff distance. Proc. Amer. Math. Soc. {\bf95},  653-658 (1985)

\bibitem{fcaa1} Cernea, A.:   On some fractional differential inclusions with random parameters. Fract. Calc. Appl. Anal. {\bf21}(01), 190–199  (2018) https://doi.org/10.1515/fca-2018-0012

\bibitem{SS} Chandra, S., Abbas, S.:  Box dimension of mixed Katugampola fractional integral of two-dimensional continuous functions. Fract. Calc. Appl. Anal. {\bf25}(4),  1022–1036 (2022) https://doi.org/10.1007/s13540-022-00050-2
\bibitem{SSS} Chandra, S., Abbas, S.:  Analysis of fractal dimension of mixed Riemann-Liouville integral. Numer. Algorithms {\bf91}, 1021–1046 (2022)
\bibitem{SS1} Chandra, S., Abbas, S.,  Liang, Y. S.: A note on fractal dimension of Riemann–Liouville fractional integral. Fractals {\bf32}(02), 2440001 (2024)

\bibitem{hess} Christian, H.:  Set-Valued Integration and Set-Valued Probability Theory: An Overview, Handbook of Measure Theory, 619-673 (2002)

\bibitem{anca}  Croitoru, A., Mesiar, R., Sambucini, A.R.,  Satco, B.:  Special Issue on Set-Valued Analysis (2021). Mathematics {\bf10}, 2703 (2022) 10.3390/math10152703

\bibitem{debreu} Debreu, G.:  Integration of correspondences, Proceedings of the Fifth Berkeley Symposium on Mathematical Statistics and Probability, University of California Press, 351-372 (1966) 

\bibitem{huang} Huang, Y., Lv, J.,   Liu, Z.:  Existence results for Riemann-Liouville fractional evolution inclusions. Miskolc Math. Notes {\bf17}(1), 305-325 (2016) 

\bibitem{kara} Kara, H., Ali, M.A., Budak, H.: Hermite-Hadamard-Mercer type inclusions for interval-valued functions via Riemann-Liouville fractional integrals. Turkish J. Math. {\bf46}(6), 2193–2207 (2022)

\bibitem{frac2} Khan, H.  Ahmed, S.,  Alzabut, J., Azar, A.T.:  A generalized coupled system of fractional differential equations with application to finite time sliding mode control for Leukemia therapy. Chaos Solitons and Fractals {\bf174}, 113901 (2023) 

\bibitem{CK}  Kuratowski, C.: Topologie I: Monografie Matematyczne. 20 Polish Mathematical Society, Warsaw (1948) 

\bibitem{L1} Liang, Y.S.:  Box dimensions of Riemann-Liouville fractional integrals of continuous functions of bounded variation. Nonlin. Anal. {\bf72}(11),  4304-4306 (2010). 
\bibitem{L2}  Liang, Y.S.: Fractal dimension of Riemann-Liouville fractional integral of 1-dimensional continuous functions. Fract. Calc. Appl. Anal. {\bf21}(6), 1651–1658 (2019). https://doi.org/10.1515/fca-2018-0087
\bibitem{li2}Li, H., Huang, C., Li, T.:  Dynamic complexity of a fractional-order predator-prey system with double delays. Phys. A. {\bf 526}, 120852 (2019)
\bibitem{lup} Lupulescu, V.: Fractional calculus for interval-valued functions. Fuzzy Sets and Systems {\bf 265}, 63-85 (2015)
\bibitem{GM}  M\"{a}gerl, G.: A unified approach to measurable and continuous selections. Trans. Amer. Math. Soc. {\bf245}, 443-452 (1978)
\bibitem{VK} Kiryakova, V. Generalized Fractional Calculus and Applications. Wiley, New York (1994)
\bibitem{VK1}  Kiryakova, V.: (Convenor of Round Table Discussion): A long standing conjecture failed? In: Transform Methods \& Special Functions’, Varna ’96 (Proc. 2nd Internat. Workshop), pp. 584–593, Inst. Math.
Inform. – Bulg. Acad. Sci., Sofia (1998)
\bibitem{VK2}  Kiryakova, V.: Generalized fractional integrals based on the Rathie I-function as a kernel. Nonlinear Dynamics {\bf113}, 34423–34437 (2025) https://doi.org/10.1007/s11071-025-11723-x 
\bibitem{pod} Podlubny, I.:  Fractional Differential Equations. An Introduction to Fractional Derivatives, Fractional Differential Equations, to Methods of Their Solution and Some of Their Applications. Math. Sci. Engrg., 198, Academic Press, New York (1999)
 \bibitem{ergen} Riley, D.M., Stathas, A.,  Gutiérrez-Oribio, D., Stefanou, I.:  Reinforcement learning-based adaptive time-integration for nonsmooth dynamics. Machine Learning for Computational Science and Engineering {\bf 1}(2), 44  (2025)
 
 
\bibitem{shen} Shen, Y.: The Cauchy-type problem for interval-valued fractional differential equations with the Riemann-Liouville gH-fractional derivative. Adv. Differ. Equ. {\bf10}, 1-13 (2016)

\bibitem{sriv} Srivastava, H.M., Sahoo, S.K., Mohammed, P.O., Kodamasingh, B.,  Hamed, Y.S.: New Riemann–Liouville fractional order inclusions for convex functions via interval-valued settings asociated with pseudo-order relations. Fractal and Fractional {\bf6}(4), 212 (2022)

\bibitem{VS} Staicu, V.:  Continuous selections of solutions sets to evolution equations. Proc. Amer. Math. Soc. {\bf113}, 403-413 (1991) 

\bibitem{stef}  Stefanini, L., Bede, B.:  Generalized Hukuhara differentiability of interval-valued functions and interval differential equations. Nonlinear Anal. {\bf71}, 1311–1328 (2009)
\bibitem{Tuba}  Tunç, T.: Hermite-Hadamard type inequalities for interval-valued fractional integrals with respect to another function. Math. Slovaca {\bf72}, 1501–1512 (2022)
\bibitem{MV}  Verma, M., Priyadarshi, A., Verma, S.: Vector-valued fractal functions: Fractal dimension and fractional calculus. Indag. Math. {\bf34}(4), 830–853 (2023) 
\bibitem{SV} Verma  S., Viswanathan, P.:  A note on Katugampola fractional calculus and fractal dimensions. Appl. Math. Comput. {\bf339}, 220–230 (2018)
\bibitem{SV1} Verma, S.,  Viswanathan, P.:  Bivariate functions of bounded variation: Fractal dimension and fractional integral. Indag. Math. {\bf31}(2), 294–309 (2020)
\bibitem{BY}  Yu, B.Y.,  Liang,  Y.S.: Approximation with continuous functions preserving fractal dimensions of the Riemann–Liouville operators of fractional calculus. Fract. Calc. Appl. Anal. {\bf26},  2805-2836 (2023).  https://doi.org/10.1007/s13540-023-00215-7
\bibitem{BY1} Yu, B.Y.,  Selmi, B., Liang, Y.S.:  General fractal dimensions of graphs of continuous functions associated with the Katugampola fractal integral. Fractals, {\bf33}(05), 2550040 (2025)   
\bibitem{frac1} Zada, A.,  Ali, S.,    Li,  T.: Analysis of a new class of impulsive implicit sequential fractional differential equations, Int. J. Nonlinear Sci. Numer. Simul. {\bf21}(6), 571-587 (2020) 
\end{thebibliography}

 \end{document}